\documentclass[oneside,reqno]{amsart}
\usepackage{amssymb,amsmath}
\usepackage[english]{babel}
\usepackage{amsfonts}
\usepackage{color}
\usepackage{amsthm}
\usepackage[colorlinks=true,linkcolor=NavyBlue, citecolor=NavyBlue,urlcolor=NavyBlue]{hyperref}
\usepackage{graphicx}
\usepackage{mathtools}
\usepackage[usenames,dvipsnames]{pstricks}
\usepackage{bm}
\usepackage{amsmath}
\usepackage{listings}
\usepackage{hyperref}
\usepackage{tikz}
\usepackage{tikz-cd}
\usepackage{cancel}

\usepackage{xy}
\usepackage[draft]{fixme}
\DeclareMathOperator{\GL}{GL}

\DeclareMathOperator{\syc}{SC}

\newcommand*{\pmat}[1]{\begin{bmatrix}#1\end{bmatrix}}

\DeclareMathOperator{\Ann}{Ann}
\DeclareMathOperator{\gl}{GL}

\DeclareMathOperator{\rank}{Rank}

\DeclareMathOperator{\syp}{Sp}

\let\phi\varphi
\theoremstyle{plain}
\newtheorem{theorem}{Theorem}[section]

\newtheorem{proposition}[theorem]{Proposition}

\theoremstyle{remark}
\newtheorem{remark}[theorem]{Remark}

\theoremstyle{definition}
\newtheorem{definition}[theorem]{Definition}
\newtheorem{example}[theorem]{Example}
\newtheorem*{notation*}{Notation}

\usepackage{listings}
\lstdefinelanguage{GAP}{%
 morekeywords={%
 Assert,Info,IsBound,QUIT,%
 TryNextMethod,Unbind,and,break,%
 continue,do,elif,%
 else,end,false,fi,for,%
 function,if,in,local,%
 mod,not,od,or,%
 quit,rec,repeat,return,%
 then,true,until,while%
 },%
 sensitive,%
 morecomment=[l]\#,%
 morestring=[b]",%
 morestring=[b]',%
}[keywords,comments,strings]

\usepackage[T1]{fontenc}
\usepackage[variablett]{lmodern}
\usepackage{xcolor}
\usepackage[utf8]{inputenc}
\usepackage{amsmath}
\usepackage{amsthm}
\usepackage{amssymb}
\usepackage[dvipsnames]{xcolor}
\usepackage{mathtools}
\usepackage{verbatim} 
\usepackage{bm}
\usepackage{hyperref}
\usepackage{mathdots}
\usepackage{nicematrix}
\usepackage{tikz}
\usepackage{xfrac}
\usepackage{booktabs}
\usepackage{faktor}
\usepackage{listings}
\usepackage{lstlang0}
\usepackage{xcolor}

\definecolor{codegreen}{rgb}{0,0.6,0}
\definecolor{codegray}{rgb}{0.5,0.5,0.5}
\definecolor{codepurple}{rgb}{0.58,0,0.82}
\definecolor{backcolour}{rgb}{0.95,0.95,0.92}

\lstdefinestyle{mystyle}{
	backgroundcolor=\color{backcolour},   
	commentstyle=\color{codegreen},
	keywordstyle=\color{magenta},
	numberstyle=\tiny\color{codegray},
	stringstyle=\color{codepurple},
	basicstyle=\ttfamily\footnotesize,
	breakatwhitespace=false,         
	breaklines=true,                 
	captionpos=b,                    
	keepspaces=true,                 
	numbersep=5pt,                  
	showspaces=false,                
	showstringspaces=false,
	showtabs=false,                  
	tabsize=2
}

\usepackage[T1]{fontenc}

\usepackage[utf8]{inputenc}
\usepackage{amsmath}
\usepackage{amsthm}
\usepackage{amssymb}
\usepackage[dvipsnames]{xcolor}
\usepackage{mathtools}
\usepackage{verbatim} 
\usepackage{bm}
\usepackage{hyperref}
\usepackage{mathdots}
\usepackage{nicematrix}
\usepackage{tikz}
\usepackage{xfrac}
\usepackage{booktabs}
\usepackage{faktor}
\usepackage{listings}
\usepackage{lstlang0}
\usepackage{xcolor}

\definecolor{codegreen}{rgb}{0,0.6,0}
\definecolor{codegray}{rgb}{0.5,0.5,0.5}
\definecolor{codepurple}{rgb}{0.58,0,0.82}
\definecolor{backcolour}{rgb}{0.95,0.95,0.92}

\lstdefinestyle{mystyle}{
	backgroundcolor=\color{backcolour},   
	commentstyle=\color{codegreen},
	keywordstyle=\color{magenta},
	numberstyle=\tiny\color{codegray},
	stringstyle=\color{codepurple},
	basicstyle=\ttfamily\footnotesize,
	breakatwhitespace=false,         
	breaklines=true,                 
	captionpos=b,                    
	keepspaces=true,                 
	numbersep=5pt,                  
	showspaces=false,                
	showstringspaces=false,
	showtabs=false,                  
	tabsize=2
}

\usepackage[T1]{fontenc}

\usepackage{enumerate}
\DeclareMathAlphabet{\mbb}{U}{bbold}{m}{n}

\usepackage{mathtools}
\usepackage{orcidlink}
\usepackage{yfonts}

\begin{document}

\title{Classification of small binary bibraces via bilinear maps} 
\author[R.~Civino]{Roberto Civino 
\orcidlink{0000-0003-3672-8485}
}
\author[V.~Fedele]{Valerio Fedele}

\address{DISIM \\
 Universit\`a degli Studi dell'Aquila\\
 via Vetoio\\
 67100 Coppito (AQ)\\
 Italy}       

\email{roberto.civino@univaq.it}
\email{valerio.fedele@guaduate.univaq.it}

\date{} 


\begin{abstract}
We classify small binary bibraces, using the correspondence with alternating algebras over the field $\mathbb F_2$, up to dimension eight, also determining their isomorphism classes. These finite-dimensional algebras, defined by an alternating bilinear multiplication and nilpotency of class two, can be represented by subspaces of skew-symmetric matrices, with classification corresponding to $\GL(m, \mathbb F_2)$-orbits under congruence. Our approach combines theoretical invariants, such as rank sequences and the identification of primitive algebras, with computational methods implemented in Magma. These results also count the number of possible alternative operations that can be used in differential cryptanalysis.
\end{abstract}

\maketitle

\section{Introduction}

The study of algebraic structures related to braces and their generalisations has seen rapid development in recent years, largely due to their central role in the theory of set-theoretic solutions of the Yang--Baxter equation. Since the seminal works of Rump~\cite{Ru07}, and Guarnieri and Vendramin~\cite{GV16}, it has become clear that braces provide a natural framework for constructing and analysing such solutions. In this context, binary bibraces represent a particular family of structures that can be described equivalently in terms of binary alternating algebras over $\mathbb{F}_2$, that is, finite-dimensional algebras with alternating bilinear multiplication and nilpotency of class two~\cite{civino2025binary}. This correspondence allows one to turn combinatorial problems on bibraces into linear-algebraic ones, specifically into the study of subspaces of skew-symmetric matrices over $\mathbb{F}_2$ up to congruence.

The classification of these algebras is not only of independent algebraic interest but also motivated by applications. In fact, the problem of counting the number of admissible binary alternating algebras has been mentioned in different works~\cite{brunetta2019hidden,calderini2021properties}, both in the algebraic and in the cryptographic contexts, as a challenging enumeration problem. Similar enumeration problems have been considered in other related algebraic settings as well, for instance on counting and classification of Hopf--Galois structures~\cite{byott2007hopf,MR4047559,darlington2025algorithm,darlington2025hopf}. Furthermore, even if not explicitly mentioned in the remainder of this paper, our results provide in particular a counting argument for elementary abelian subgroups of the affine group that mutually normalise with the translation group. This interpretation follows from the equivalence between binary alternating algebras and certain alternative group operations defined on vector spaces over~$\mathbb{F}_2$, and highlights an additional group-theoretic contribution of our classification~\cite{civino2025binary}.

In cryptographic applications, alternative group operations derived from bibraces have been employed in the study of differential attacks~\cite{biham1991differential,wagner2004towards}, where they can replace the usual XOR operation in the propagation of differences~\cite{civino2019differential,baudrin2023commutative,calderini2024differential,calderini2024optimal,baudrin2025commutative}. While this connection will be mentioned only briefly here, it provides a concrete motivation: the parameters of these algebras quantify the number of admissible alternative operations, and therefore the extent of possible attack strategies. In particular, the cases of dimension four and eight are especially relevant, since they cover the most common sizes of S-boxes used in block ciphers~\cite{biham1998serpent,daemen2002design,bogdanov2007present}.

The main difficulty of the classification problem lies in the utter combinatorial explosion of possibilities as the dimension increases. The problem translates into understanding the orbits of the action of $\mathrm{GL}(m, \mathbb{F}_2)$ on $d$-dimensional subspaces of the space of $m \times m$ skew-symmetric matrices. While conceptually straightforward, this orbit problem is computationally prohibitive: the number of subspaces grows extremely rapidly, and direct computation quickly becomes infeasible. For example, already in dimension eight, the number of candidate structures exceeds billions, and naive orbit enumeration clashes with the complexity of the problem.

To overcome these difficulties, we employ a twofold strategy. On the one hand, we use algebraic invariants, such as rank sequences of matrix spaces, to partition the set of subspaces into coarser equivalence classes, thereby reducing the search space. On the other hand, we use the computational algebra system Magma~\cite{bosma1997magma} to carry out explicit orbit computations when possible, and group-theoretic arguments when direct computations are infeasible. These tools, while not sufficient to fully classify higher-dimensional cases in general, prove to be powerful enough to obtain a complete classification in small cases.

In particular, we extend the enumeration results of binary alternating algebras (and hence bibraces) to all structures of dimension at most eight, providing criteria for primitivity and rank-based invariants that help distinguish non-isomorphic cases. These results solve, in the range up to dimension eight, the problem of counting such structures mentioned in previous works.
At the same time, from the cryptographic perspective, our results clarify the landscape of possible alternative operations for differential attacks, showing that the range of admissible structures of dimensions four and eight (the dimensions most relevant for S-boxes) can be completely enumerated and understood. While it has long been known that there are 105 operations on 4-bit blocks, here we find that there are 117,833,335,446,015 operations on 8-bit blocks, corresponding to roughly 47 bits of entropy for the search of an alternative operation for differential cryptanalysis.\\

The paper is organised as follows. In Section~\ref{sec:preliminaries} we recall the necessary background on binary alternating algebras, their correspondence with binary bibraces, and the matrix representation of the induced bilinear maps. Section~\ref{sect:count} develops the classification techniques, introducing invariants such as rank sequences and the notion of primitivity, needed to obtain the complete classification up to dimension eight summarised in Table~\ref{tab:num}.

\section{Preliminaries and Background}
\label{sec:preliminaries}

This section collects the essential definitions, notation, and previous results on binary bibraces and binary alternating algebras that will be used in the sequel. All algebraic structures are considered over the field~$\mathbb{F}_2$. Here we follow the notation of a previous paper where bibraces were singled out as a particular construction and given a name~\cite{civino2025binary}. The classification results rely and are inspired by a broader line of research on more general structures, carried out by many authors over the years.~\cite{Ca05,CR09,aragona2019regular,CCS19,Ch19,CARANTI2020647,DC23}.
\begin{definition}\label{def:bibra}
A \emph{binary bibrace} is a triple $(R, +, \circ)$ where:
\begin{itemize}
    \item $(R, +)$ and $(R, \circ)$ are both elementary abelian $2$-groups (vector spaces over~$\mathbb{F}_2$);
    \item the operations satisfy the \emph{biskew brace} identities:
    \[
    (x+y)\circ z=x\circ z+z+y\circ z, 
    \quad x\circ y+z=(x+z)\circ z\circ(y+z)
    \]
    for all $x, y, z \in R$.
\end{itemize}
\end{definition}
The previous identities are equivalent to the \emph{dual normalisation} meaning that the group $T_{\circ}(R) = \{ t_x : y \mapsto y \circ x \mid x \in R \}$ of $\circ$-translations normalises the group $T_{+}(R)$ of $+$-translations and vice versa.

\begin{definition}\label{def:alg}
A \emph{binary alternating algebra} is an $\mathbb{F}_2$-algebra $(R,+,\cdot)$
satisfying
\begin{enumerate}
    \item \emph{alternating law:} $x \cdot x = 0$ for all $x \in R$;
    \item \emph{nilpotency of class two:} $x \cdot y \cdot z = 0$ for all $x, y, z \in R$.
\end{enumerate}
The \emph{annihilator} of $R$ is the subspace
$
\Ann(R) = \{ a \in R \mid a \cdot R  = 0 \}
$, while the \emph{square} is
$
R^2 = \langle x \cdot y \mid x, y \in R \, \rangle.
$
The nilpotency condition implies $R^2 \subseteq \Ann(R)$.
\end{definition}


The algebraic structure of Definitions~\ref{def:bibra} and~\ref{def:alg} are equivalent, meaning that, given a binary alternating algebra $(R, +, \cdot)$, the operation
\[
x \circ y = x + y + x \cdot y
\]
defines a binary bibrace on $R$. Conversely, any binary bibrace $(R,+,\circ)$ corresponds to a binary alternating algebra of nilpotency class two~\cite{civino2025binary}, equipped with the product defined by
\[x\cdot y=x+y+x\circ y.\] 
This equivalence implies that the classification of binary alternating algebras directly determines the possible bibrace structures with the same support, i.e.\ all alternative operations for differential cryptanalysis under the dual normalisation constraint.

\subsection{Representation via skew-symmetric matrices and bilinear maps}
Given a nondegenerate alternating bilinear map $\phi : V \times V \longrightarrow W$
over $\mathbb{F}_2$-vector spaces $V$ and $W$, it is possible to define a product on the direct sum $V\oplus W=R$
that equips it with the structure of a binary alternating algebra $(R,+,\cdot)$ such that $\Ann(R)=0_V\oplus W$, where
\[(x_1,y_1)\cdot(x_2,y_2)=(0,\phi(x_1,x_2))\]
for all $x_1,x_2\in V$, $y_1,y_2\in W$ (here \emph{alternating} means $\phi(x,x)=0$ for all $x\in V$).

Conversely, the product of a binary alternating algebra $(R,+,\cdot)$ induces an alternating
bilinear map
\[\widehat{\phi}: R\times R\longrightarrow \Ann(R),\quad \widehat{\phi}(x,y)=x\cdot y\]
and its restriction to a complement $V$ of the annihilator $W=\Ann(R)$ is a nondegenerate alternating bilinear map.

Let us consider the finite-dimensional case and let $(v_1,\dots,v_m,w_1,\dots,w_d)$ be an ordered basis of $R$, where $v_1, \dots , v_m$ and $w_1,\dots,w_d$ are bases of $V$ and $W$, respectively. Then the bilinear map $\widehat{\phi}$ defined over the entire space is determined by a sequence $(\widehat{B}_1,\dots,\widehat{B}_d)$ of $d$ binary square $(m+d)\times (m+d)$ matrices. More precisely, we can write
		\[\widehat{\phi}(x,y)=\phi_1(x,y)w_1+\dots+\phi_d(x,y)w_d,\]
		where 
		\[\phi_k: R\times R\longrightarrow \mathbb F_2,\quad \phi_k(x,y)=x\widehat{B}_ky^t\]
		is a bilinear form for $1\leq k\leq d$. Conversely, if $\phi_1,\dots,\phi_d$ are given, then the matrix $\widehat{B}_k$ is defined by
		\[\widehat{B}_k=\left[\begin{array}{c|c}\phi_k(v_i,v_j)&0\\\hline 0&0\end{array}\right],\quad 1 \leq i,j \leq m.\]
		We denote by $B_k$ the $m\times m$ submatrix of $\widehat{B}_k$ whose $(i,j)$ entry $B_k[i,j]$ is given by $v_i\widehat{B}_kv_j^t$ for each $1 \leq i,j \leq m$ and $1 \leq k \leq d$. Then the matrices $B_k$ are the $m\times m$ matrices associated to the bilinear map $\phi:V\times V\longrightarrow W$.
		
		It is proved that $B_1,\dots,B_d$ are \emph{skew-symmetric} (symmetric and zero-diagonal)~\cite{civino2025binary}. Moreover, it is shown that
		
		\begin{equation}\label{R2}
			\dim R^2=\dim\langle B_1,\dots,B_d\rangle.
		\end{equation}
		
Conversely, starting from a sequence $B_1,\dots,B_d$ of binary skew-symmetric matrices of size $m$, one can define a binary alternating algebra by endowing the vector space $\mathbb{F}_2^m\oplus \mathbb{F}_2^d$ with the product  \[(x_1,y_1)\cdot(x_2,y_2)=(0,(x_1B_1x_2^t,\dots,x_1B_dx_2^t)).\]
Ensuring the nondegeneracy of the bilinear map $\phi(x_1,x_2)=(x_1B_1x_2^t,\dots,x_1B_dx_2^t)$ requires only that the matrix formed by horizontally concatenating $B_1,...,B_d$ has a rank of $m$.
We remark that the nondegeneracy condition is equivalent to requiring that the annihilator of the associated algebra coincides with the subspace $0_m\oplus\mathbb{F}_2^d$.

\begin{definition}
Let $V=\langle v_1,\dots, v_m\rangle$ and $W=\langle w_1,\dots,w_d\rangle$ be two finite-dimensional vector spaces over $\mathbb F_2$, and let $R=V\oplus W$. Given a nondegenerate bilinear map $\phi:V\times V\longrightarrow W$, we denote by $\cdot_\phi$ the algebra product induced by $\phi$ over $R$ as 
\begin{align*}\label{product}
		(x_1,y_1)\cdot(x_2,y_2)=(0,\phi(x_1,x_2)).
	\end{align*}
 The matrices $B_1,\dots,B_d$ satisfying
\[\phi(x_1,x_2)=(x_1B_1x_2^t,\dots,x_1B_dx_2^t)\quad \text{for } x_1,x_2\in V\]
are called the \emph{defining matrices of $(R,+,\cdot_\phi)$}.
\end{definition}

The matrix space generated by the defining matrices of a binary alternating algebra is a subspace of the space
\[
\Lambda_m = \{ m \times m \ \text{skew-symmetric matrices over} \ \mathbb{F}_2 \}.
\]

To determine the isomorphism classes of small binary alternating algebras, we use the fact that two such algebras are isomorphic if and only if their corresponding $d$-dimensional subspaces
are in the same orbit under the action of $\mathrm{GL}(m, \mathbb{F}_2)$ by congruence. The classification of binary alternating algebras is equivalent to the classification of $\mathrm{GL}(m, \mathbb{F}_2)$-orbits of $d$-dimensional subspaces of $\Lambda_m$, which we address in the remainder of the paper.

	\section{Number of binary structures}\label{sect:count}
	
	Let $R$ be a vector space of dimension $m+d$ over $\mathbb{F}_2$ and let $W$ be a subspace of $R$ of dimension $d$. The number of binary alternating algebras that one can impose on the vector space $R$ such that $\Ann(R)=W$ is given by the number of the length-$d$ ordered sequences of skew-symmetric matrices over $\mathbb{F}_2$ whose horizontal concatenation has rank $m$~\cite{civino2025binary}, i.e.
	\[s_{m,d}=\#\left\{(B_1,\dots,B_d):\ B_i\in\Lambda_m,\ \rank\pmat{B_1,\dots,B_d}=m\right\}.\]
	Since 
	\[t_{m,d}=\frac{(2^{m+d}-1)(2^{m+d}-2)\dots(2^{m+d}-2^{d-1})}{(2^{d}-1)(2^{d}-2)\dots(2^{d}-2^{d-1})},\]
	i.e.\ the Gaussian binomial coefficient $\binom{m+d}{d}_2$,
	counts the number of $d$-dimensional subspaces of $R$, which corresponds to the possible choices for the annihilator of the algebra, the total number of such algebras is given by the product $s_{m,d} \cdot t_{m,d}$.
	In Table~\ref{tab:num}, we indicate $s_{m,d}$ and $t_{m,d}$ for each couple of parameters $(m,d)$ within the range $3\leq m+d=n\leq 8$.
	In the remainder of the paper, we show how we computed $s_{m,d}$ in the relevant cases. We recall that not all values of $d = \dim(\Ann(R))$ are admissible for a binary alternating algebra of a given total dimension. Indeed, by construction one always has $R^2 \subseteq \Ann(R)$, so that $1 \leq d \leq \dim(R)-2$, and additional restrictions arise from the alternating property of the bilinear map. In particular, when the total dimension of $R$ is even, the case 
$d=1$ cannot occur. 
As a consequence, some values of $d$ are forbidden and therefore do not appear in Table~\ref{tab:num}.
	
	\begin{table}
	\[
	\begin{matrix*}
		\hline
		n& m& d& \text{\# $d$-dimensional spaces}	& \text{\# operations over standard basis}\\
		\hline								        
		\hline								        
		3& 2& 1& 7 									&1				\\
		\hline								        
		\hline								        
		4& 2& 2& 35 								&3				\\
		\hline								        
		\hline								        
		5& 4& 1& 31 								&28				\\
		\hline
		5& 3& 2& 155 								&42				\\
		\hline
		5& 2& 3& 155 								&7				\\
		\hline								        
		\hline								        
		6& 4& 2& 651 								&3360				\\
		\hline
		6& 3& 3& 1395 								&462				\\
		\hline
		6& 2& 4& 651 								&15				\\
		\hline								        
		\hline								        
		7& 6& 1& 127 								&13888				\\
		\hline
		7& 5& 2& 2667 								&937440				\\
		\hline
		7& 4& 3& 11811 								&254968				\\
		\hline
		7& 3& 4& 11811 								&3990				\\
		\hline
		7& 2& 5& 2667 								&31				\\
		\hline								        
		\hline								        
		8& 6& 2& 10795 								&1012435200				\\
		\hline
		8& 5& 3& 97155 								&1065765120				\\
		\hline
		8& 4& 4& 200787 							&16716840				\\
		\hline
		8& 3& 5& 97155 								&32550				\\
		\hline
		8& 2& 6& 10795 								&63				\\
		\hline
	\end{matrix*}
	\]
		\caption{Number of small binary structures}
	\label{tab:num}
	\end{table}
	%

	%
	\subsection{Primitive binary alternating algebras}
	
	Let $(R,+,\cdot_\phi)$ be a binary alternating algebra of dimensional parameters $m$ and $d$. Assume that $R^2$ is a proper subspace of $\Ann(R)$, i.e. $\Ann(R)=R^2\oplus U$ for some nonzero subspace $U$ of dimension $1\leq k\leq d-1$. Then $U$ is a null subalgebra of $R$ and the quotient algebra $R/U$ is isomorphic to an algebra of dimensional parameters $m$ and $d-k$.
	
	Conversely, starting from an algebra with a defining sequence $[B_1,\dots,B_d]$ such that $\Ann(R)=R^2$, it is easy to construct an algebra of dimensional parameters $m$ and $d+k$, for example by extending the sequence with $k$ zero matrices
	\[[B_1,\dots,B_d,0,\dots,0].\]
	Therefore, for classification purposes, it is convenient to use $\dim(R^2)$ and $m$ as parameters. Since the dimension of the annihilator and the counting of operations are relevant in cryptographic applications, we distinguish the two classification approaches by introducing the following definition.

	\begin{definition}
		Let $(R,+,\cdot_\phi)$ be a binary alternating algebra, let $W=\Ann(R)$ and let $V$ be a complement of $W$. We say that $R$ is primitive if $\Ann(R/R^2)=0$ or, equivalently, $W=R^2$. 
	\end{definition}
In what follows, we denote by $E_{ij}$ the standard basis of $\Lambda_m$, i.e. for $1\leq i<j\leq m$,
	\[E_{ij}=e_{ij}+e_{ji},\]
	where $e_{ij}$ denotes an elementary matrix ($e_{ij}[i,j]=1$ and $e_{ij}[h,k]=0$ for $(i,j)\neq(h,k)$).

	\begin{proposition}
		Let $(R,+,\cdot_\phi)$ be a primitive binary alternating algebra, $R=V\oplus R^2$, $V\simeq \mathbb{F}_2^m$, $R^2\simeq \mathbb{F}_2^d$. The following statements hold:
		\begin{enumerate}
			\item $1+(m\mod 2)\leq \dim(R^2)\leq \dim(\Lambda_m)$;
			\item if $\dim(R^2)=1$, then the isomorphism class of $R$ is unique among binary alternating algebras with underlying vector space $\mathbb{F}_2^m\oplus\mathbb{F}_2^d$ ($m$ even);
			\item if $m=3$, then $R$ belongs to one of the two isomorphism class identified by the following sequences of matrices:
			\[(E_{12},E_{13},E_{23}),\quad (E_{12},E_{23}).\]
		\end{enumerate}
	\end{proposition}
	\begin{proof}
		Let $B_1,\dots,B_d\in\Lambda_m$ be the defining matrices of $(R,+,\cdot_\phi)$. By Eq.~\eqref{R2}, $\dim(R^2)=\dim\langle B_1,\dots,B_d\rangle\leq\dim(\Lambda_m)$. Since the bilinear map $\phi$ is nondegenerate, $\rank\left(\pmat{B_1&\dots&B_d}\right)=m$. Now, if $d=\dim(R^2)=1$, then $\rank(B_1)=m$ and so $m$ is even. On the other hand, if $m$ is odd,  $\rank(B)\leq m-1$ for each $B\in\Lambda_m$ and so $d\geq 2$. 
		
		The uniqueness of the isomorphism class in the case $d=1$ follows by the fact that any two skew-symmetric matrices of the same rank are congruent.
		
		Finally, for $m=3$, an easy check shows that the seven $2$-dimensional subspaces of $\Lambda_3$ are pairwise congruent. Thus, the only two isomorphism classes of primitive algebras are determined by a sequence of matrices that generates the entire space $\Lambda_3$ when $d=3$ or a $2$-dimensional subspace when $d=2$. 
	\end{proof}
Table~\ref{tab:class} shows the classification of binary alternating algebras with underlying vector space $\mathbb{F}_2^m\oplus\mathbb{F}_2^d$ up to dimension $n=m+d=8$. For each pair of parameters $m$ and $d$, we indicate the number of isomorphism classes of binary structures, or, equivalently, the number of congruence classes of $k$-dimensional subspaces of $\Lambda_m$, where $k\leq d$. Moreover, we indicate the isomorphism classes of primitive algebras (i.e. the number of congruence classes of $d$-dimensional subspaces of $\Lambda_m$).
	
	\begin{table}
	\begin{align*}
		\begin{matrix*}
			\hline
			$n$ &$m$&$d$					&\text{\# classes} & \text{\# primitive}			\\
			\hline								                 
			\hline								                 
			*	& 2 & \geqslant 1 			&1 &1												\\
			\hline								                 
			*    & 3& \geqslant 3 & 2		&2													\\
			\hline								                 
			\hline								                 
			5	& 3 & 2 					&1 &1												\\
			\hline								                 
			5	& 4 & 1 					& 1 &1												\\
			\hline								                 
			\hline								                 
			6	& 4 & 2 					& 4 &3												\\
			\hline								                 
			\hline								                 
			7	& 4	& 3 					& 9 &5												\\
			\hline								                 
			7	& 5 & 2						& 2 &2												\\
			\hline								                 
			7	& 6 & 1						& 1 &1												\\
			\hline								                 
			\hline								                 
			8	& 4 & 4						& 13&4												\\
			\hline								                 
			8	& 5 & 3						& 18&16												\\
			\hline								                 
			8	& 6 & 2						& 9 &8 												\\
			\hline
	\end{matrix*}\end{align*}
		\caption{Isomorphism classes in small cases}
	\label{tab:class}
	\end{table}
	
	This classification was obtained through direct computation, via Magma, of the congruence classes of the skew symmetric matrix subspaces. More precisely, we computed the action by congruence of $\gl(m,\mathbb{F}_2)$ on the set of $k$-dimensional subspaces, $k\leq d$.
	For the last case, $m=6,\ d=2$, we used a different approach because this type of direct computation would have required an excessive amount of time and virtual memory. Thus, we identified appropriate representatives of the congruence classes and calculated the cardinality of each class as the ratio between the order of $\gl(6,2)$, \[\#\gl(6,\mathbb{F}_2)=\prod_{j=0}^{5}{(2^6-2^k)}=20158709760,\] and the order of the \emph{self-congruence  group} of a space $\mathcal{B}\subseteq\Lambda_m$, which we will introduce in the following sections along with the details of the procedure. 	
	\subsection{Ranks criterion}\label{sec:ranks}
The tool we are about to introduce allows one to determine a non-isomorphism criterion for alternating algebras.

	\begin{definition}
 We call \emph{(ascendant) sequence of (matrix) ranks} of a $k$-dimensional matrix vector space $\mathcal{B}\subseteq \mathbb{F}_2^{m\times m}$ any sequence of length $2^k-1$ of the matrix ranks $(r_1,\dots,r_{2^k-1})$, where $r_i=\rank(B_i)$, $0\neq B_i\in \mathcal{B}$ and such that
	\[r_1\leq r_2\leq\dots\leq r_{2^k-1}.\] 
	\end{definition}
	
To introduce the ranks criterion, we need the following result.

\begin{theorem}[\cite{civino2025binary}]\label{congruent}
Let $(R,+,\cdot_\phi)$ and $(S,+,\cdot_\psi)$ be two nilpotent algebras of class two over $\mathbb F_2$, with underlying vector space $\mathbb F_b^m\oplus \mathbb F_2^d$ and defining matrices $B_1,\dots,B_d$ and $C_1,\dots, C_d$ respectively. Then   $R$ and $S$ are isomorphic algebras if and only if $\langle B_1,\dots,B_d\rangle$ and $\langle C_1,\dots, C_d\rangle$ are congruent (i.e.\ there exists an invertible matrix $A$ such that
		$A\langle C_j\rangle_jA^t=\langle B_i\rangle_i$).
\end{theorem}

	\begin{theorem}\label{ranks}
		Let  $R=\mathbb{F}_2^m\oplus \mathbb{F}_2^d$, $(R,+,\cdot_\phi)$ and $(R,+,\cdot_\psi)$ be two nilpotent algebras of class two with defining matrices $B_1,\dots,B_d$ and $C_1,\dots,C_d$ respectively. If the matrix spaces $\langle B_i\rangle_i$ and $\langle C_i\rangle_i$ have different sequence of ranks, then the two algebras are not isomorphic.
	\end{theorem}	
	\begin{proof}
		The sequence of ranks of the space $\langle B_i\rangle_i$ is preserved under the action $A\langle B_i\rangle_iA^t$ for each $A\in\gl(m,\mathbb{F}_q)$. Therefore if $\langle B_i\rangle_i$ and $\langle C_i\rangle_i$ have different sequences of ranks, the two spaces are not congruent and the claim follows by Theorem \ref{congruent}.
	\end{proof}
	
	The following counterexample shows that, in general, two binary alternating algebras with the same sequence of ranks are not isomorphic.
	
	\begin{example}\label{example}
		Let us consider the following two $3$-dimensional matrix vector subspace of $\Lambda_4$, 
		\[\mathcal{B}=\langle E_{14},E_{23},E_{13}+E_{24}\rangle,\quad\mathcal{C}=\langle E_{12},E_{23},E_{13}+E_{24}\rangle.\]
		By a direct computation with Magma, we can verify that every $2$-dimensional subspace of $\mathcal{B}$, respectively of $\mathcal{C}$, has one of the following rank chains
		\[(2,2,4),\,(2,4,4),\,(4,4,4),\]
		respectively
		\[(2,2,2),\,(2,4,4).\]
		Now, for each $A\in\gl(4,\mathbb{F}_2)$, the space $A\langle E_{14},E_{23}\rangle A^t$ has rank chain $(2,2,4)$ and so it is not contained in $\mathcal{C}$. Thus, $\mathcal{B}$ is not congruent to $\mathcal{C}$. 
		However, both have the same rank chain as shown in Table~\ref{tab:ranks}.\\
		
		\begin{table}
		\[\begin{tabular}{@{}ccc}
			$\rank$& $\mathcal{B}$					&$\mathcal{C}$\\
			\hline	\hline
			2 & $E_{14}$						& $E_{12}$\\
			\hline
			2 & $E_{23}$						& $E_{23}$\\
			\hline					
			2 & $E_{14}+E_{23}+E_{13}+E_{24}$                 & $E_{12}+E_{23}$\\
			\hline
			4 & $E_{13}+E_{24}$                 & $E_{13}+E_{24}$\\
			\hline
			4 & $E_{14}+E_{13}+E_{24}$          & $E_{12}+E_{13}+E_{24}$\\
			\hline
			4 & $E_{23}+E_{13}+E_{24}$          & $E_{23}+E_{13}+E_{24}$\\
			\hline
			4 & $E_{14}+E_{23}$   & $E_{12}+E_{23}+E_{13}+E_{24}$\\
			\hline
			
		\end{tabular}\]
		\caption{Ranks and subspaces of Example~\ref{example}}
		\label{tab:ranks}
		\end{table}
		
		We conclude this section by observing that the set of $k$-dimensional subspaces of $\Lambda_m$ identified by a fixed sequence of ranks is partitioned into congruence classes. Therefore, the partition induced by the rank sequences on the set of $k$-dimensional subspaces is a coarser partition than that induced by the congruence classes. However, it simplifies the identification of suitable representatives of the congruence classes when the number of spaces is very large.
	\end{example}
	
	\subsection{Binary alternating algebras with $2$-dimensional annihilator}
	
	In this section, we will address the case of $2$-dimensional spaces of $\Lambda_m$ with the aim of classifying binary alternating algebras of dimension $8$ with a $2$-dimensional annihilator. Indeed, given that the cardinality of $\gl(6,2)$ is very large, it is computationally infeasible to use the method used for previous cases, which involves directly calculating the group's action on the set of subspaces.
	
	For a given $2$-dimensional subspace $\mathcal{B}$ of $\Lambda_6$, the results presented here have allowed us to compute the cardinality of the congruence class of $\mathcal{B}$ as the ratio of $\#\gl(6,2)$ to the cardinality of the \emph{self-congruence group} of $\mathcal{B}$, which is defined as follows.
	
	\begin{definition}\label{selfcongruence}
		Let $\mathcal{B}\subseteq\Lambda_m$ be a subspace of skew symmetric matrices over $\mathbb{F}_2$. We define 
		\[\syc(\mathcal{B})=\{\,A:\, A\in\gl(m,\mathbb{F}_2)\ |\ A\mathcal{B} A^t=\mathcal{B}\,\}\]
		as the \emph{self-congruence group} of the space $\mathcal{B}$.
	\end{definition}
	
	Clearly, this method requires that determining $\#\syc(\mathcal{B})$ is computationally feasible and that it is possible to identify a set of representatives for the congruence classes. For this last requirement, we have used the rank criterion introduced in Section~\ref{sec:ranks}, along with the fact that the cardinalities of the congruence classes are distinct (in pairs) and that their sum coincides with the total number of $2$-dimensional spaces of $\Lambda_6$, namely
	\[\frac{(2^{15}-1)(2^{15}-2)}{6}.\]
	
We illustrate the results in Tables~\ref{tab:nondeg} and~\ref{tab:deg}. The first $8$ congruence classes contain $168732256$ subspaces. The horizontal concatenation of any two non zero matrices $B_1, B_2, B_3=B_1+B_2$ in each of these spaces gives rise to a rank $6$ matrix. Indeed, for each $i,j\in\{1,2,3\}$, $i\neq j$, \[\rank\pmat{B_i&B_j}=\rank\pmat{B_1&B_2&B_3},\]
	because the span of the columns of $B_3=B_1+B_2$, $C(B_1+B_2)$ is contained in the span of $C(B_1)\cup C(B_2)$. In particular, the bilinear map associated with $(B_i, B_j)$ is nondegenerate, and the congruence class is an isomorphism class of alternating binary algebras. Clearly, the number of these algebras corresponds to the number of ordered bases of each space, which is $6$ times the number of subspaces, i.e. $1012393536$. Now, the total number of algebras with a $2$-dimensional annihilator is given by \[1012393536+3\times13888=1012435200,\]
	where $13888$ is the number of skew symmetric matrices $B$ with a rank of $6$ and $3$ is the number of defining matrices sequences of length $2$ which yield a $1$-dimensional $R^2$ space, i.e. $(B,0), (0,B), (B,B)$.
	
	The subspaces contained in the last $6$ classes are related to degenerate bilinear maps, i.e. the horizontal concatenation of two matrices taken in any of these subspaces, has rank less than $6$.\\
	
		\begin{table}
	\[\begin{matrix}
		\hline
		\text{Class}&	\text{Ranks} & \text{Cardinality}\\
		\hline
		\hline
		C_1     &	$(6,6,6)$ 	 & 13332480\\
		\hline
		C_2     & $(4,6,6)$	 & 27998208\\
		\hline
		C_3     & $(4,6,6)$	 & 26248320\\
		\hline
		C_4     & $(2,6,6)$	 & 2187360\\
		\hline
		C_5     & $(4,4,6)$	 & 69995520\\
		\hline
		C_6     & $(2,4,6)$	 & 4666368\\
		\hline
		C_7     & $(4,4,4)$	 & 15554560\\
		\hline
		C_8     &	$(4,4,4)$	 & 8749440\\
		\hline
	\end{matrix}\]
	
	\caption{Spaces related to nondegenerate bilinear maps}
	\label{tab:nondeg}
	\end{table}
	
	\begin{table}
	\[\begin{matrix*}
		\hline
		\text{Class}&	\text{Ranks}        &\text{Cardinality}\\
		\hline
		C_{9}&	$(4,4,4)$			&6562080\\
		\hline
		C_{10}&	$(4,4,4)$   		&36456\\
		\hline
		C_{11}&	$(2,4,4)$			&73728\\
		\hline
		C_{12}&	$(2,4,4)$   		&3072\\
		\hline
		C_{13}&	$(2,2,4)$			&182280\\
		\hline
		C_{14}&	$(2,2,2)$			&9765\\
		\hline
	\end{matrix*}\]
	\caption{Spaces related to degenerate bilinear maps}
	\label{tab:deg}
	\end{table}

We conclude this section by illustrating the results used for the classification. Given a sequence $S$ of matrices defined over a field
of characteristic different from 2, the Magma function \textbf{IsometryGroup($S$)} returns the
isometry group of the system, which we have defined as the self-congruence group of Definition~\ref{selfcongruence}.
In characteristic 2, the algorithm for computing such group is not implemented in Magma; to overcome this, we relied on Proposition~\ref{classification}, which we are now going to prove. Here $\syp$ denotes the symplectic group. The following proposition shows that if two 2-dimensional spaces $\langle B_1,B_2\rangle$, $\langle C_1,C_2\rangle$ are congruent, then the two subgroup $\syp(B_1)\cap\syp(B_2)$ and $\syp(C_1)\cap\syp(C_2)$ are conjugated in $\gl(m,\mathbb{F}_2)$. Notice that $\syp(B_1)\cap\syp(B_2)<\syp(B_1+B_2)$, indeed for every $X\in \syp(B_1)\cap\syp(B_2)$,
	\[X(B_1+B_2)X^t=XB_1X^t+XB_2X^t=B_1+B_2.\]
	In particular, $\syp(B_1)\cap\syp(B_2)=\syp(B_1)\cap\syp(B_2)\cap\syp(B_1+B_2)$.
	
	\begin{proposition}\label{conjcongr}
		Let $\langle B_1,B_2\rangle,\langle C_1,C_2\rangle\in\Lambda_m$ be $2$-dimensional spaces of skew symmetric matrices over $\mathbb{F}_2$. If they are congruent, then $\syp(B_1)\cap\syp(B_2)$ is conjugate to $\syp(C_2)\cap\syp(C_2)$.
	\end{proposition}
	\begin{proof}
		Assume that $\langle B_1,B_2\rangle,\langle C_1,C_2\rangle\in\Lambda_m$ are congruent. Then there exist $Z\in\gl(m,\mathbb{F}_2)$ such that
		\begin{align*}
			ZC_1Z^t=a_{11}B_1+a_{12}B_2,\\
			ZC_2Z^t=a_{21}B_1+a_{22}B_2,
		\end{align*}
		for some invertible matrix $A=\pmat{a_{11}&a_{12}\\a_{21}&a_{22}}$. Now, let $X\in \syp(B_1)\cap\syp(B_2)$. Then, for $i\in\{1,2\}$, 
		\[XZC_iZ^tX^t=X(a_{i1}B_1+a_{i2}B_2)X^t=a_{i1}B_1+a_{i2}B_2=ZC_iZ^t.\]
		Thus, $Z^{-1}XZ\in\syp(C_1)\cap\syp(C_2)$, i.e. $Z^{-1}(\syp(B_1)\cap\syp(B_2))Z\subseteq\syp(C_1)\cap\syp(C_2)$. The opposite inclusion follows by the symmetry of the congruence relation.
	\end{proof}
	
	\begin{remark}
		In the hypotheses of the previous proposition, we explicitly observe that
		given an invertible matrix $Z\in\gl(m,\mathbb{F}_2)$ and setting $G_B=\syp(B_1)\cap\syp(B_2)$, $G_C=\syp(C_1)\cap\syp(C_2)$, then
		\[Z\langle C_1,C_2\rangle Z^t=\langle B_1,B_2\rangle \quad\implies\quad Z^{-1}G_BZ=G_C.\]
	\end{remark}
	
	\begin{proposition}\label{classification}
		Let $\mathcal{B}=\langle B_1,B_2\rangle\subseteq\Lambda_m$ be a $2$-dimensional subspace of skew-symmetric matrices over $\mathbb{F}_2$, let $G=\syp(B_1)\cap\syp(B_2)$ and let $N=N_{\gl(m,\mathbb{F}_2)}(G)$. Then the following statements hold:
		\begin{enumerate}
			\item \label{item1}  $G\leq\syc(\mathcal{B})\leq N$ (in particular, $\syc(\mathcal{B})=\{A\in N\ |\ A\mathcal{B} A^t=\mathcal{B}\}$);
			\item  if $T$ is a right transversal for $G$ in $N$, and $S$ is the subset of $T$ formed by the matrices $A$ such that $A\mathcal{B}A^t=\mathcal{B}$, then $\syc(\mathcal{B})=\langle G\cup S\rangle$;
			\item $[\,\gl(m,\mathbb{F}_2)\,:\,\syc(\mathcal{B})\,]$ is the number of $2$-dimensional subspaces of $\Lambda_m$ congruent to $\mathcal{B}$.
		\end{enumerate}
	\end{proposition}
	\begin{proof}
		Clearly $G\leq \syc(\mathcal{B})$. Let $Z\in\syc(\mathcal{B})$. By Proposition \ref{conjcongr}, $Z^{-1}GZ=G$ and so $Z\in N$.
		
		Now, it follows by construction that $\langle G\cup S\rangle\subseteq \syc(\mathcal{B})$. To prove the opposite inclusion, let us consider $Z\in \syc(\mathcal{B})$ and observe that, by (\ref{item1}), $Z\in N$. So $Z$ is contained in a right coset of $G$ in $N$, i.e. $Z=Z_GZ_N$ for some $Z_G\in G$ and $Z_N\in N$. Now, $\mathcal{B}=Z\mathcal{B}Z^t=Z_GZ_N\mathcal{B}Z_N^tZ_G^t$. This means that $\mathcal{B}=Z_G^{-1}\mathcal{B}(Z_G^{-1})^t=Z_N\mathcal{B}Z_N^t$. Thus, $Z_N\in GZ_S$ for some $Z_S\in S$, i.e. $Z=Z_G'Z_S$, $Z_G'\in G$.
Finally, it is clear that two invertible matrices, belonging to the same right coset of $\syc(\mathcal{B})$, yield the same matrix space by acting on $\mathcal{B}$. So the index of $\syc(\mathcal{B})$ in $\gl(m,\mathbb{F}_2)$ coincides with the number of spaces which are congruent to $\mathcal{B}$.
	\end{proof}
	
	Thanks to item~(\ref{item1}) of the previous proposition, it is possible to restrict the search for invertible matrices such that $A\mathcal{B}A^t=\mathcal{B}$ to the normaliser of $G$ in $\operatorname{GL}(m,\mathbb{F}_2)$. In particular, it suffices to check the self-congruence property for the matrices belonging to a system of representatives of the right cosets of $G$ in $N$. This is particularly convenient in the case $m=6$.
	
\section*{Statements and Declarations}
R.\ Civino is member of INdAM-GNSAGA and is supported by the Centre of EXcellence on Connected, Geo-Localized and 
 Cybersecure Vehicles (EX-Emerge), funded by Italian Government under CIPE resolution n.\ 70/2017 (Aug.\ 7, 2017). 
 R. Civino thankfully acknowledges support by MUR-Italy via PRIN 2022RFAZCJ `Algebraic methods in
cryptanalysis'.\\
The authors have no relevant financial or nonfinancial interests to disclose. R.~Civino and V.~Fedele wrote the main manuscript text. All
authors reviewed the manuscript. \\
Data availability statements: no datasets were generated or analyzed during the current study, and therefore no data are available to be shared. 

\bibliographystyle{amsalpha}
\bibliography{biblio}


\end{document}